\documentclass[12pt]{amsart}
\usepackage[utf8]{inputenc}
\usepackage{amsbsy,amssymb,amsfonts}
\usepackage{mathtools}
\usepackage[top=3cm,left=2.5cm,right=2.5cm,bottom=3cm]{geometry}
\usepackage{enumerate}
\usepackage[bookmarks=true,hyperindex,pdftex,colorlinks,citecolor=red, linkcolor=blue]{hyperref}

\DeclareMathOperator{\Orb}{Orb}
\DeclareMathOperator{\Span}{span}

\newtheorem{theorem}{Theorem}[section]
\newtheorem{lemma}[theorem]{Lemma}
\newtheorem{proposition}[theorem]{Proposition}

\theoremstyle{definition}
\newtheorem{example}[theorem]{Example}
\newtheorem{question}{Question}
\newtheorem*{BCPquestion}{Question}

\def\RR{\mathbb R}
\def\NN{\mathbb N}

\def\CC{\mathbb C}
\def\KK{\mathbb K}
\def\A{\mathcal A}
\def\L{\mathcal L}
\def\B{\mathcal B}

\usepackage{todonotes}

\title{Existence of hypercyclic algebras in Fréchet algebras}
\author{Fernando Costa Jr.}
\author{Álvaro Rocha}

\address[F. Costa Jr. and A. Rocha]{Universidade Federal da Paraíba - Campus I, Departamento de Matemática, Jardim Universitário, s/n, Bairro Castelo Branco, CEP 58051-900, João Pessoa, Brazil}
\email{fernando@mat.ufpb.br} 

 \allowdisplaybreaks

\begin{document}


\begin{abstract}
    We prove that every separable infinite-dimensional Fréchet algebra $X$ admits a continuous linear operator $T$ supporting a dense invariant hypercyclic algebra, giving an affirmative answer to a question of Bayart, Costa Jr. and Papathanasiou. In fact, every dense countable-dimensional subalgebra $\mathcal A$ of $X$ can be prescribed as an invariant hypercyclic algebra. When $X$ admits a continuous norm, the operator can additionally be chosen in the form $T=I+K$, where $K$ is nuclear, so that $\mathcal A = \operatorname{span} \operatorname{Orb}(a,T)$ for any prescribed $a\in\mathcal A\backslash\{0\}$.
\end{abstract}

\maketitle

\section{Introduction}\label{sec:introduction}

Linear dynamics studies the orbits of continuous linear operators on topological vector spaces. An operator $T:X\to X$ is called \emph{hypercyclic} if there exists $x\in X$ whose orbit
\[
  \Orb(x,T):=\{T^n x:n\geq 0\}
\]
is dense in $X$. Such an $x$ is called a \emph{hypercyclic vector}, and the set of all hypercyclic vectors for $T$ is denoted by $HC(T)$. We refer to the monographs \cite{BMbook,GEPbook} for the general theory.

The existence of hypercyclic operators is well understood for Fréchet spaces. Every separable infinite-dimensional Banach space admits a hypercyclic operator, as Ansari \cite{Ansari1997} and Bernal-González \cite{Bernal1999} proved independently in response to a question of Rolewicz \cite{Rolewicz1969}. Bonet and Peris \cite{BonetPeris1998} extended this conclusion to non-normable Fréchet spaces. Since a dense orbit forces separability and a nonzero finite-dimensional space cannot support a hypercyclic operator, these results characterize the nonzero Fréchet spaces on which such operators exist.

A related problem concerns the structure of the set of hypercyclic vectors. Every hypercyclic operator on a Fréchet space admits a dense invariant linear subspace whose nonzero vectors are hypercyclic (see \cite{Herrero1991, Bourdon1993, Bes1999}). A \emph{hypercyclic subspace} for $T$ is a closed infinite-dimensional subspace contained in $HC(T)\cup\{0\}$. Although not every hypercyclic operator has such a subspace, every separable infinite-dimensional Banach space supports an operator that does (see \cite{LeonMontes1997}). Following the results for Fréchet spaces with a continuous norm from \cite{Bernal2006,Petersson2006} and for $\omega=\KK^\NN$ from \cite{BesConejero2006}, Menet \cite[Theorem 2.8]{Menet2013} established the existence of a mixing operator with a hypercyclic subspace on every separable infinite-dimensional Fréchet space.

When $X$ carries an algebra structure, one can ask for a subalgebra inside $HC(T)\cup\{0\}$. A nonzero subalgebra $\mathcal A\subset X$ is called a \emph{hypercyclic algebra} for $T$ if $\mathcal A\setminus\{0\}\subset HC(T)$. Subalgebras are not required to contain the identity of the ambient algebra. For pointwise multiplication on $H(\mathbb C)$, the differentiation operator supports a hypercyclic algebra \cite{BMbook,Shkarin2010}, whereas nontrivial translations do not \cite{Aron2007}. Thus the existence of hypercyclic operators on an algebra leads to a further existence problem. Bayart, Costa Jr. and Papathanasiou \cite[Question 6.6]{BCP2021} asked the following question.

\begin{BCPquestion}
Does every separable infinite-dimensional Banach algebra support an operator with a hypercyclic algebra?
\end{BCPquestion}

We answer this question affirmatively in the more general setting of Fréchet algebras.

\begin{theorem}\label{corol:teo:princ}
Every separable infinite-dimensional Fréchet algebra admits a continuous linear operator $T$ and a dense $T$-invariant hypercyclic algebra of algebraic dimension $\aleph_0$.
\end{theorem}

In fact, our Theorem \ref{teo:general:prescription} shows that every dense countable-dimensional subalgebra can be prescribed as an invariant hypercyclic algebra.  The key observation is that such a subalgebra, endowed with the topology inherited from $X$, is a countably dimensional metrizable locally convex space whose completion is $X$. Shkarin’s result \cite[Theorem 1.5]{Shkarin2014} therefore applies, and the resulting operator extends uniquely and continuously to $X$.  This resolves the hypercyclic-algebra existence question in the full setting of separable infinite-dimensional Fréchet algebras.

When the ambient algebra admits a continuous norm, Theorem \ref{teoprincipal} gives additional control over the operator and its orbit. Given a dense countable-dimensional subalgebra $\A$ and a nonzero vector $a\in\A$, the operator can be chosen in the form $T=I+K$, with $K$ nuclear, so that
\[
    \A=\Span\Orb(a,T).
\]
This refinement combines Albanese's prescribed-orbit construction \cite{Albanese2011} with the Herrero--Bourdon theorem \cite{Herrero1991,Bourdon1993,Bes1999}.

Section \ref{sec:preliminaries} collects the definitions and the results used in the paper. Section \ref{sec:results} establishes the general prescription result and its refinement under a continuous norm. Section \ref{sec:questions} discusses closed hypercyclic algebras and the extension to nonlocally convex $F$-algebras.

\section{Preliminaries}\label{sec:preliminaries}

All vector spaces are over $\KK\in\{\RR,\CC\}$ and are assumed Hausdorff. We write $\NN=\{1,2,\ldots\}$, and all dimensions are algebraic (Hamel) dimensions.

An \emph{$F$-space} is a complete metrizable topological vector space. A \emph{Fréchet space} is a locally convex $F$-space; its topology can be defined by a separating increasing sequence $(\Vert\cdot\Vert_n)_n$ of seminorms for which $X$ is complete in the metric \[d(x,y)=\sum\limits_{n=1}^\infty\frac{1}{2^n}\min(1,\Vert x-y\Vert_n).\]
If there exists a norm $\Vert\cdot\Vert:X\to\RR$ that is continuous in the topology of $X$, we say that $X$ \emph{admits a continuous norm}.

For a topological vector space $X$, we denote by $\L(X)$ the set of all continuous linear operators $T:X\to X$ and simply call them \emph{operators}.

A \emph{topological algebra} is a topological vector space $X$ that is equipped with a bilinear and associative multiplication $\cdot : X\times X \to X$ that is continuous in the product topology. In particular, when $X$ is a Fréchet space (or an $F$-space), we say that $X$ is a \emph{Fréchet algebra} (or \emph{$F$-algebra}). We say that a topological algebra $X$ is \emph{unital} when there exists $1_X\in X$ such that \[1_X\cdot x=x\cdot 1_X=x\] for all $x\in X$. The element $1_X$ is said to be the \emph{unit} of $X$. For simplicity, we will write $x\cdot y$ as $xy$.

Let $D\subset X$. The \emph{subalgebra generated} by $D$ is \[\A(D):=\Span\{d_1\cdots d_k : k\geq 1, \, d_1,\dots,d_k\in D\}.\]
If $D$ is countable, we say that $\A(D)$ is a \emph{countably generated} subalgebra.

The general prescription result relies on the following theorem of Shkarin.

\begin{theorem}[{\cite[Theorem 1.5]{Shkarin2014}}] \label{teo:shkarin:countable}
Let $E$ be a metrizable locally convex space of countably infinite algebraic dimension. There exists $S\in\L(E)$ such that every nonzero vector of $E$ is hypercyclic for $S$.
\end{theorem}

We shall also use the standard fact that every continuous linear map from a dense linear subspace $E$ of a Fréchet space $X$ into $X$, with $E$ carrying the inherited topology, extends uniquely to a continuous linear operator on $X$ (see \cite[Ch. 1, Exercise 19]{RudinBook}).

Under a continuous-norm hypothesis, the next result provides a more precise prescribed-orbit construction. It was first proved by Grivaux in \cite{Grivaux2003} for separable Banach spaces. It was subsequently generalized by Albanese in \cite{Albanese2011} for separable Fréchet spaces with a continuous norm.

\begin{theorem}[{\cite[Theorem 3.5]{Albanese2011}}] \label{teoalbanese}
Let $X$ be a separable infinite-dimensional Fréchet space which admits a continuous norm. Let $V=\{v_n : n\geq1\}$ be a dense set of linearly independent vectors of $X$. Then there is an operator $T\in\mathcal{L}(X)$ of the form $T=I+K$, with $K$ a nuclear operator on $X$, such that $\Orb(v_1,T)=V$.
\end{theorem}

The next result was proved independently by Herrero in \cite{Herrero1991} and by Bourdon in \cite{Bourdon1993} for complex Hilbert spaces. In the same paper, Bourdon states that this result is valid in the Banach setting. In \cite{Bes1999}, Bès provides a proof for real Fréchet spaces and mentions that the result also holds in the complex case.

\begin{theorem}[Herrero-Bourdon] \label{teobourdon}
Let $X$ be a separable infinite-dimensional Fréchet space and let $T\in\L(X)$. If $x\in X$ is a hypercyclic vector for $T$, then \[\{p(T)x : p\mbox{ is a polynomial}\}\backslash\{0\}\] is a dense set of hypercyclic vectors. In particular, any hypercyclic operator admits a dense invariant subspace consisting, except for zero, of hypercyclic vectors. 
\end{theorem}

\section{Results}\label{sec:results}

We first establish the elementary facts needed for the existence result and the subsequent orbit-prescription theorem.

\begin{lemma}\label{lemma1preliminar}
    Every separable Fréchet algebra $X$ contains a dense subalgebra $\A_0$ of at most countable algebraic dimension. If $X$ is infinite-dimensional, then $\dim\A_0=\aleph_0$. If $X$ has a unit, then $\A_0$ can be chosen to contain $1_X$.
\end{lemma}
\begin{proof}
    Let $D$ be a countable and dense subset of $X$. If $X$ is unital, we can assume $1_X\in D$. Define $\A_0=\A(D)$. Since $D\subset\A_0$, it follows that $\A_0$ is dense in $X$. Clearly, $\A_0$ is countably generated, because $D$ is countable. Thus, $\dim\A_0\leq\aleph_0$. If $X$ is infinite-dimensional, then $\A_0$ cannot be finite-dimensional, otherwise it would be closed (since $X$ is Hausdorff), which gives $\A_0=\overline{\A_0}=X$, which is impossible. Thus $\dim \A_0=\aleph_0$. Finally, if $X$ has a unit $1_X$, then $1_X\in D\subset \A_0$ so the algebra $\A_0$ is unital.
\end{proof}

The following lemma is close to \cite[Lemma 4.3]{SchenkeShkarin2013} and can be demonstrated using a very similar argument. We provide a proof here for the sake of completeness.

\begin{lemma}\label{lema2preliminar}
    Let $X$ be a separable infinite-dimensional Fréchet space. 
    Every dense subspace $E\subset X$ contains a countable, dense, linearly independent subset of $X$. 
\end{lemma}
\begin{proof}
    Let $E\subset X$ be a dense subspace and $(U_n)_n$ be a countable basis of nonempty open sets of $X$. Choose a nonzero $e_1\in E\cap U_1$ and consider $F_1=\Span\{e_1\}$. Then $F_1$ is a closed proper subspace of $X$. In particular, $\operatorname{int}F_1=\varnothing$. Thus, $U_2\not\subset F_1$ and hence $U_2\setminus F_1$ is a nonempty open set. Since $E$ is dense, there is $e_2\in E\cap(U_2\setminus F_1)$. Now, suppose that $e_1,\cdots, e_{n-1}$ have been chosen with $e_i\in E\cap U_i$ and $e_i\notin F_j$, for all $j=1,\cdots,i-1$ and for all $i=2,\cdots,n-1$. Let $F_{n-1}=\Span\{e_1,\cdots,e_{n-1}\}$. Similarly, $F_{n-1}$ is a closed and proper subspace of $X$. Thus, $U_n\not\subset F_{n-1}$ and $U_n\setminus F_{n-1}$ is a nonempty open set. Hence, there is $e_n\in E\cap(U_n\setminus F_{n-1})$. This defines a sequence $(e_n)_{n\geq1}$ such that $e_n\notin F_{n-1}$, for all $n\geq 2$. It follows that $(e_n)_n$ is a linearly independent set. Moreover, $\{e_j\}_{j\geq1}\cap U_n\neq\varnothing$, for all $n\geq1$. Therefore, $(e_n)_n$ is dense in $X$.
\end{proof}
\begin{lemma}\label{lemma3preliminar}
    Let $X$ be a separable infinite-dimensional Fréchet space. Every dense subspace $E\subset X$ admits a Hamel basis $\B$ that is dense in $X$. 
\end{lemma}

\begin{proof}
    By Lemma \ref{lema2preliminar},  $E$ admits a linearly independent and dense subset in $X$, say $M$. We can extend $M$ to a Hamel basis $\B$ of $E$. Since $M$ is dense in $X$, this implies that $\B$ is dense in $X$.
\end{proof}

We first obtain the general prescription result as a consequence of Shkarin's theorem.

\begin{theorem}\label{teo:general:prescription}
Let $X$ be a separable infinite-dimensional Fréchet algebra and $\A\subset X$ be a dense countable-dimensional subalgebra. There exists $T\in\L(X)$ such that
\[
    T(\A)\subset\A
    \qquad\text{and}\qquad
    \A\setminus\{0\}\subset HC(T).
\]
\end{theorem}
\begin{proof}
Endow $\A$ with the topology inherited from $X$. By Theorem \ref{teo:shkarin:countable}, there exists $S\in\L(\A)$ for which every nonzero vector of $\A$ is hypercyclic. Let $T\in\L(X)$ be its continuous extension. Then $T(\A)\subset\A$, and, for every $a\in\A\setminus\{0\}$, the orbit $\Orb(a,T)=\Orb(a,S)$ is dense in $\A$, hence in $X$.
\end{proof}

The existence statement announced in the introduction follows immediately.

\begin{proof}[Proof of Theorem \ref{corol:teo:princ}]
Apply Lemma \ref{lemma1preliminar} to obtain a dense subalgebra of algebraic dimension $\aleph_0$, and then apply Theorem \ref{teo:general:prescription}.
\end{proof}

When $X$ admits a continuous norm, we can additionally prescribe a vector whose orbit spans the algebra and choose the operator as a nuclear perturbation of the identity. The following result combines Albanese's prescribed-subspace construction (in the way of \cite[Corollary 3.6]{Albanese2011}) with Theorem \ref{teobourdon} and the preceding lemmas.

\begin{theorem}\label{teoprincipal}
Let $X$ be a separable infinite-dimensional Fréchet algebra with a continuous norm. Let $\A$ be a dense countable-dimensional subalgebra of $X$ and let $a\in \A\backslash\{0\}$. Then there exists $T\in\mathcal{L}(X)$ of the form $T=I+K$, where $K\in \L(X)$ is nuclear, such that $\A=\Span \Orb(a,T)$. In particular, $\A$ is a $T$-invariant hypercyclic algebra for $T$.
\end{theorem}
\begin{proof}
    By Lemma \ref{lemma3preliminar}, $\A$ admits a Hamel basis $\B$ which is dense in $X$. Since $\A$ is countable-dimensional, $\B$ is countable. We can write $a=\sum_{j=1 }^{m}\alpha_j u_j$, for some $m\in\NN$, some distinct $u_1,\dots,u_m\in\B$ and some $\alpha_1,\dots,\alpha_m$ with $\alpha_1\neq 0$. Define $\B'=(\B\setminus\{u_1\})\cup\{a\}$. It is clear that $\B'$ is still a dense, countable, linearly independent set that spans $\A$. Writing $\B'=\{v_n\}_n$ with $v_1=a$, it follows from Theorem \ref{teoalbanese} that there exists an operator $T\in\mathcal{L}(X)$ of the form $T=I+K$, where $K$ is nuclear, such that $\Orb(a,T)=\B'$. Hence, $\A=\Span\B'=\Span\Orb(a,T)$.  In particular, $a$ is a hypercyclic vector for $T$. By Theorem \ref{teobourdon}, the algebra \[\A=\Span\Orb(a,T)=\{p(T)a : p\text{ is a polynomial}\}\]
    is a $T$-invariant set whose nonzero elements are hypercyclic.
\end{proof}

\section{Final remarks and open questions}\label{sec:questions}

The preceding results settle the existence problem for separable infinite-dimensional Fréchet algebras. We conclude with questions concerning closed hypercyclic algebras and the possible extension to the nonlocally convex setting.

Hypercyclic subspaces are closed and infinite-dimensional by definition. The existence of such structures in any separable infinite-dimensional Fréchet space (see \cite{Bernal2006,Menet2013}) motivates an analogous question for algebras. In \cite{BCP2021}, Bayart, Costa Jr. and Papathanasiou proved that no convolution operator $P(D)$ acting on $H(\CC)$ and induced by a non-constant polynomial $P$ has a closed hypercyclic algebra. They also proved that no weighted backward shifts on $\ell_p(\NN)$ or $c_0(\NN)$ with the coordinatewise product has a closed hypercyclic algebra either. This leads to the question whether every separable infinite-dimensional Banach algebra admits an operator supporting a closed hypercyclic algebra.

This problem, however, can have trivial answers in some situations. Let $X$ be a separable infinite-dimensional Fréchet algebra. If $X$ admits a finite-dimensional subalgebra $\A\neq \{0\}$, then there is an operator $T\in\L(X)$ for which $\A$ is a closed hypercyclic algebra. Indeed, it suffices to consider a basis $\{e_1,\cdots, e_n\}$ of $\A$. Hence, since $X$ is separable, there is a countable and dense subset of $X$, say $D'$. Define $D=D'\cup\{e_1,\cdots,e_n\}$. By an argument entirely analogous to that of Lemma \ref{lemma1preliminar}, the algebra $\A_0:=\A(D)$ is countable-dimensional and dense in $X$. Moreover, $\A\subset\A_0$. We now use Theorem \ref{teo:general:prescription} to prescribe $\A_0$, and thus $\A$, as a hypercyclic algebra for some operator $T\in\L(X)$. Since $\A$ is finite-dimensional, it is a closed hypercyclic algebra for $T$. Observe that having a non-trivial finite-dimensional subalgebra is always true for unital algebras, since we may take $\A=\KK 1_X$. 

Yet, not every Banach algebra contains a nonzero finite-dimensional subalgebra, as the following example shows.

\begin{example}
Endow the space $\ell_1(\NN)$ with the \emph{Cauchy product} $(x_n)\cdot (y_n) = (c_n)$, where $c_n=\sum_{k=1}^{n-1}x_{n-k}y_k.$ Then $(\ell_1(\NN),\cdot)$ is a non-unital Banach algebra that contains no non-trivial finite-dimensional subalgebra. Indeed, given $x\in \ell_1$ nonzero, it is not difficult to see that the set $\{x, x^2, x^3, \dots\}$ is linearly independent. Therefore, $\A(\{x\})$ must be infinite-dimensional.
\end{example}

The dimension discussion above motivates the following question.

\begin{question} \label{question:banach:closed:inf-dim}
    Does every separable infinite-dimensional Banach algebra admit an operator with a closed infinite-dimensional hypercyclic algebra?
\end{question}

The algebras obtained by our constructions have countable algebraic dimension, thus they cannot contain a closed infinite-dimensional subspace, since every such subspace of a Banach (or Fréchet) space has algebraic dimension at least $\mathfrak{c}$ (see \cite{PopoolaTweddle1977}). Therefore, these constructions do not settle Question \ref{question:banach:closed:inf-dim}.

Another natural question is whether our results can be extended to the context of $F$-algebras. Of course, not all $F$-spaces admit a hypercyclic operator. Indeed, as shown by Shkarin in \cite{Shkarin2012}, the space $X = L_p[0,1]\times\KK^n$, with $0< p<1$ and $n\in\NN$, is an $F$-space that supports no hypercyclic operator. This space, however, cannot admit a continuous norm, otherwise it would have to admit a hypercyclic operator. This fact can be deduced from \cite[Theorem 1.4 and Remark 1.3]{Shkarin2011}. We include a short verification here for completeness. For the definitions of $\mathfrak{M}_0$ and $\mathfrak{M}_1$, see \cite[Definition 1.2]{Shkarin2011}.

\begin{proposition}\label{prop:shkarin:continuous-norm}
    Every separable infinite-dimensional $F$-space $X$ with a continuous norm supports a hereditarily hypercyclic operator.
\end{proposition}
\begin{proof}
Any separable $F$-space belongs to $\mathfrak{M}_0$, so $X\in \mathfrak{M}_0$. Let $p$ be a continuous norm on $X$. Then $\ker p = p^{-1}(0)=\{0\}$, thus the codimension of $\ker p$ in $X$ is the dimension of $X$ itself, which is infinite. By \cite[Remark 1.3]{Shkarin2011} it follows that $X\in \mathfrak{M}_1$. Therefore, $X\in \mathfrak{M}:=\mathfrak{M}_0\cap\mathfrak{M}_1$. Applying \cite[Theorem 1.4]{Shkarin2011} with $k=1$, there is a hereditarily hypercyclic operator semigroup $\{T_t\}_{t\in\KK}$. In particular, $T=T_1$ is a hereditarily hypercyclic operator on $X$, as we wanted.
\end{proof}
Therefore, for separable infinite-dimensional $F$-algebras with a continuous norm, Proposition \ref{prop:shkarin:continuous-norm} settles the existence of hypercyclic operators, leaving the following algebraic question.

\begin{question}\label{question:falgebra:continuous-norm}
    Does every separable infinite-dimensional $F$-algebra with a continuous norm support an operator with a hypercyclic algebra?
\end{question}

\section*{Funding} 

The first author was partially supported by the Conselho Nacional de Desenvolvimento Científico e Tecnológico (CNPq, Brazil) through grants 406457/2023-9, 403964/2024-5 and 304165/2025-5. The second author acknowledges support from the Coordenação de Aperfeiçoamento de Pessoal de Nível Superior (CAPES, Brazil).

\section*{Conflict of Interest Statement}

The authors have no conflicts of interest to disclose.

\bibliographystyle{abbrv}
\bibliography{bib}

\end{document}